\documentclass{birkjour}
\usepackage{amsmath}
\usepackage{amssymb}
\usepackage{amsthm}

\usepackage{verbatim}

\usepackage[utf8]{inputenc}

\usepackage{graphicx}
\usepackage{color}
\usepackage{enumerate}

\usepackage{esint}
\usepackage[all]{xy}
\usepackage{framed}
\usepackage{bbm}
\usepackage{subcaption}

\usepackage[super]{nth}

\theoremstyle{plain}
\newtheorem{theorem}{Theorem}[section]

\newtheorem{lemma}{Lemma}[section]

\newtheorem{proposition}{Proposition}[section]

\theoremstyle{definition}

\numberwithin{equation}{section}

\newcommand{\bfx}{\mathbf{x}}

\newcommand{\bfa}{\mathbf{a}}

\newcommand{\bfr}{\mathbf{r}}
\newcommand{\bff}{\mathbf{f}}

\newcommand{\bfp}{\mathbf{p}}

\newcommand{\bfg}{\mathbf{g}}
\newcommand{\bfh}{\mathbf{h}}

\newcommand{\Hu}{\mathcal{H}u}

\newcommand{\cH}{\mathcal{H}}

\newcommand{\dd}{\mathrm{\,d}}

\newcommand{\mR}{\mathbb{R}}
\newcommand{\oO}{\overline{\Omega}}

\newcommand{\oB}{\overline{B}}

\newcommand{\bben}{\mathbbm{1}}

\DeclareMathOperator{\tr}{tr}

\DeclareMathOperator{\di}{div}

\DeclareMathOperator{\dist}{dist}

\begin{document}

\author[\normalcolor K. K. Brustad]{Karl K. Brustad}
\title{The Infinity-Potential in the Square}
\address{Frostavegen 1691\\ 7633 Frosta\\ Norway}
\email{brustadkarl@gmail.com}

\begin{abstract}
A representation formula for the solution of the $\infty$-Laplace equation is constructed in a punctured square, the prescribed boundary values being $u=0$ on the sides and $u=1$ at the centre. This so-called $\infty$-potential is obtained with a hodograph method. The heat equation is used and one of Jacobi's Theta functions appears. The formula disproves a conjecture.
\end{abstract}

\maketitle


\section{Introduction}

I shall obtain a special explicit solution of the $\infty$-Laplace Equation
\[\Delta_\infty u := u_x^2u_{xx} + 2u_xu_yu_{xy} + u_y^2u_{yy} = 0\]
in the plane. The solution is the $\infty$-potential (or the "capacitary function") for a square with the centre removed, at which boundary point the value $u = 1$ is prescribed. On the four sides $u = 0$.

\begin{figure}[h]%
	\includegraphics[width=0.45\textwidth]{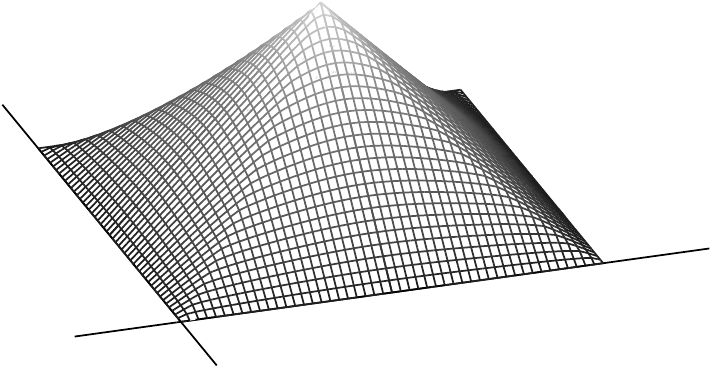}%
	\caption{The graph of the solution over the square.}%
	\label{fig:fullsolution}%
\end{figure}

Let $\Omega$ be the square $0<x<2,\,0<y<2$.
I shall show that the viscosity solution of the Dirichlet problem just described
is
\begin{equation}\label{eq:solution}
u(x,y) = \min_{\theta\in[0,\pi/2]}\max_{r\in[0,1]} \Big\{r(x\cos\theta + y\sin\theta) - W(r,\theta)\Big\}
\end{equation}
where
\begin{align*}
W(r,\theta)
	&= \frac{8}{\pi}\left(\frac{r^4}{6}\sin(2\theta) + \frac{r^{36}}{210}\sin(6\theta) + \frac{r^{100}}{990}\sin(10\theta) + \cdots\right).
\end{align*}
The formula is valid in the quadrant $0\leq x\leq 1$, $0\leq y\leq 1$,
and the obvious symmetries extends the solution to the whole $\oO$.
So far as we know, the only hitherto known explicit $\infty$-potentials are for stadium-like domains with solutions on the form $\dist(\bfx,\partial\Omega)$.

The explicit solution \eqref{eq:solution} settles a conjecture, indicated in \cite{MR1716563} and \cite{MR1886623}. A calculation shows that $u$ is \emph{not} a solution to the $\infty$-eigenvalue problem
\[\max\left\{\Lambda - \frac{|\nabla u|}{u},\; \Delta_\infty u\right\} = 0, \qquad u|_{\partial\Omega} = 0.\]
Se also \cite{MR2317552}. 

The $\infty$-Laplace equation $\Delta_\infty u = 0$ was introduced by G. Aronsson in 1967 as the limit of the $p$-Laplace equation
\[\Delta_p u := \di\left(|\nabla u|^{p-2}\nabla u\right) = 0\]
when $p\to\infty$. See \cite{MR217665}, \cite{MR237962}. For this \nth{2} order equation, the concept of viscosity solutions was introduced by T. Bhattacharya, E. DiBenedetto, and J. Manfredi in \cite{MR1155453}, and uniqueness of such solutions with continuous boundary values was proved by R. Jensen, cf. \cite{MR1218686}. For the $p$-Laplace equation, the capacitary problem in convex rings was studied by J. Lewis, \cite{MR477094}. A detailed investigation of this problem for the $\infty$-Laplace equation was given in \cite{1078-0947_2019_8_4731}, \cite{LINDGREN2021107526}.

I use the hodograph method, according to which a \emph{linear} equation is produced for a function $w = w(p,q)$ in the new coordinates
\[p = \frac{\partial u(x,y)}{\partial x},\qquad q = \frac{\partial u(x,y)}{\partial y}.\]
A solution $W(r,\theta) = w(r\cos\theta,r\sin\theta)$ is constructed with judiciously adjusted boundary values.
When transforming back we get a formula for $u$ in terms of its gradient, which, in turn, is shown to be a critical value of the objective function in \eqref{eq:solution}. The critical point -- that is, the obtained minimax $(r,\theta)$ -- is the length $r = |\nabla u|$ and the direction $\theta = \arg\nabla u$ of the gradient at $(x,y)$. 
A full account is given in Section \ref{sec:derive}.

It is interesting that the \nth{2} Jacobi Theta function
\[\vartheta_2(z,q) = 2\sum_{k=1}^\infty q^{(k-1/2)^2}\cos((2k-1)z)\]
appears in the formulas. For example, we shall see that
\[u = \frac{4}{\pi}\int_0^{\arg\nabla u}\vartheta_2\left(2\psi,|\nabla u|^{16}\right)\dd\psi,\qquad 0\leq\arg\nabla u\leq \pi/2,\]
and that the determinant of the Hessian matrix of $u$ is
\[\det\Hu = - \frac{\pi^2}{16}|\nabla u|^4\, \vartheta_2^{-2}\left(2\arg\nabla u,|\nabla u|^{16}\right),\qquad \arg\nabla u\neq (2k-1)\pi/4.\]
An immediate consequence of the last identity is that $u$ is not $C^2$ on the diagonals of the square. Indeed, the symmetries yield $\arg\nabla u(x,x) = \pi/4$ and $\vartheta_2$ is zero for odd integer multiples of $z=\pi/2$. 

Write
\[D_1 := \{(r,\theta)\;|\; 0<r<1,\,0<\theta<\pi/2\}.\]
Define $W\colon \overline{D}_1 \to\mR$,
as
\begin{equation}\label{eq:Wdef}
W(r,\theta)
:= \frac{8}{\pi}\sum_{n=1}^\infty \frac{r^{m_n^2}}{\left(m_n^2 - 1\right)m_n}\sin\left(m_n\theta\right),\qquad m_n = 4n-2,
\end{equation}
and let $u\colon\oO\to\mR$ be the extension of formula \eqref{eq:solution} to $\oO$.
I prove
\begin{theorem}\label{thm}
The function $u$ is the unique viscosity solution of the problem
\begin{equation}\label{eq:Dirprob1}
\begin{cases}
\Delta_\infty u = 0\qquad &\text{in $\Omega\setminus\{(1,1)\}$},\\
	u = 0 & \text{on $\partial\Omega$,}\\
	u = 1 &\text{at $(1,1)$,}
\end{cases}
\end{equation}
in the square $\Omega = \{(x,y)\;|\; 0<x<2,\,0<y<2\}$. Moreover, $u\in C^1(\oO\setminus\{(1,1)\})$ and $u$ is real-analytic, except at the diagonals and medians.
\end{theorem}
Actually, we do not know how smooth $u$ is across the medians.

\section{Proof of the Theorem}

The coefficients in $W$ and the exponents on $r$ are chosen so that $W(1,\theta)$ is the Fourier series for the odd $\pi$-periodic extension of $\cos\theta + \sin\theta - 1$, and so that $rW_r + W_{\theta\theta} = 0$. Thus, $W$ is the solution of the problem
\begin{equation}\label{eq:PDEpolar}
\begin{cases}
rW_r + W_{\theta\theta} = 0\qquad &\text{in $D_1$},\\
W(r,0) = W(r,\pi/2) = 0, & 0 \leq r\leq 1,\\
W(0,\theta) = 0,\\
W(1,\theta) = \cos\theta + \sin\theta - 1, & 0\leq\theta\leq\pi/2.
\end{cases}
\end{equation}

I prove first that \eqref{eq:solution} is the viscosity solution to the Dirichlet problem
\begin{equation}\label{eq:Dirproblem2}
\begin{cases}
\Delta_\infty u = 0\qquad &\text{in $\Omega_1$},\\
u(0,t) = u(t,0) = 0, &\\
u(1,t) = u(t,1) = t &\text{for $0\leq t\leq 1$,}
\end{cases}
\end{equation}
where
\[\Omega_1 := \{(x,y)\;|\; 0<x<1,\,0<y<1\}.\]
We know in advance that the solution must be linear on the medians.
It is then straight forward to show that the gluing \eqref{eq:full_solution} of the translations and rotations of \eqref{eq:solution} is the solution of the original problem \eqref{eq:Dirprob1}.

\begin{figure}[h]%
	\center
	\includegraphics{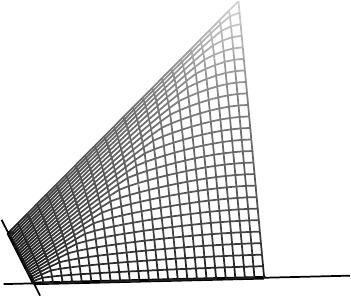}%
	\caption{The solution to problem \eqref{eq:Dirproblem2} over the square $\oO_1 = [0,1]^2$.}%
	\label{fig:omega1sol}%
\end{figure}

For each $(x,y)\in \oO_1$ I denote the \emph{objective function} $f_{(x,y)}\colon \overline{D}_1\to\mR$ in \eqref{eq:solution} by
\[f_{(x,y)}(r,\theta) := r(x\cos\theta + y\sin\theta) - W(r,\theta).\]
Note that we immediately have
\[u(x,y) \geq \min_{\theta\in[0,\pi/2]}f_{(x,y)}(0,\theta) = 0.\]
The bounds in the Proposition below show that the correct boundary values are obtained.

\begin{proposition}\label{prop:bounds_double}
For all $(x,y)\in\oO_1$,
\[1 - \sqrt{(1-x)^2 + (1-y)^2} \leq u(x,y) \leq \dist( (x,y),\partial\Omega).\]
\end{proposition}

\begin{proof}
\begin{align*}
u(x,y) &= \min_{\theta\in[0,\pi/2]}\max_{r\in[0,1]}f_{(x,y)}(r,\theta)\\
       &\geq \min_{\theta\in[0,\pi/2]}f_{(x,y)}(1,\theta)\\
       &= \min_{\theta\in[0,\pi/2]}\left\{x\cos\theta + y\sin\theta - (\cos\theta + \sin\theta - 1)\right\}\\
       &= 1 - \max_{\theta\in[0,\pi/2]}\left\{ (1-x)\cos\theta + (1-y)\sin\theta\right\}\\
       &= 1 - \sqrt{(1-x)^2 + (1-y)^2}\max_{\theta\in[0,\pi/2]}\cos\left(\theta - \phi\right)
\end{align*}
where the last line is due to a standard trigonometric identity. The lower bound is confirmed, since the maximum is at $\theta = \phi := \arctan\frac{1-y}{1-x}\in[0,\pi/2]$.

Since $W\geq 0$ (see Lemma \ref{lem:1}),
\begin{align*}
u(x,y) &= \min_{\theta\in[0,\pi/2]}\max_{r\in[0,1]}f_{(x,y)}(r,\theta)\\
       &\leq \min_{\theta\in[0,\pi/2]}\max_{r\in[0,1]}\left\{r(x\cos\theta + y\sin\theta)\right\}\\
       &= \min_{\theta\in[0,\pi/2]}\left\{x\cos\theta + y\sin\theta\right\}\\
	   &= \min\left\{x,y\right\} = \dist( (x,y),\partial\Omega).
\end{align*}
\end{proof}

%
%



I prove next that $u$ is $\infty$-harmonic in the viscosity sense in $\Omega_1$.
The strategy is to first
show that $u$ is a classical solution in $\Omega_1$ except on the diagonal. At the diagonal it is necessary to demonstrate that the smoothness breaks down in such a way that no test function can touch $u$ from below. Yet, it is essential to know that $u$ is $C^1$, because we need the gradient of a test function to align with the diagonal when the touching is from \emph{above}. The proof is then completed by showing that $x\mapsto u(x,x)$ is convex.

It turns out that the Heat Equation is helpful.
Consider an insulated rod of length $\pi/2$ with initial temperature $v(0,\theta)=1$, for $0<\theta<\pi/2$. Suppose the rod is subjected to the heat equation $v_t = v_{\theta\theta}$ with boundary conditions $v(t,0) = v(t,\pi/2) = 0$ when $t>0$. The solution of this textbook example is
\[v(t,\theta) = \frac{8}{\pi}\sum_{n=1}^\infty\frac{e^{-m_n^2 t}}{m_n}\sin(m_n\theta),\qquad m_n = 4n-2.\]
It becomes
\begin{equation}\label{eq:Udef}
U(r,\theta) := rW_r(r,\theta) - W(r,\theta) = \frac{8}{\pi}\sum_{n=1}^\infty\frac{r^{m_n^2}}{m_n}\sin(m_n\theta)
\end{equation}
under the substitution
\[r = e^{-t}.\]
Since $W\in C(\overline{D}_1)$ with $W(1,0) = W(1,\pi/2) = 0$, an immediate consequence is that $W_r(r,\theta)$ has jumps from 0 to 1 at the two corner points $r=1$, $\theta = 0$ and $r=1$, $\theta = \pi/2$. Otherwise it is continuous up to the boundary $\partial D_1$.

The temperature $v(t,\theta)$ is \emph{strictly} decreasing. To see this note that $h := v_t$ is again caloric with initial values $h(0,\theta)\leq 0$ and lateral values $h(t,0) = h(t,\pi/2) = 0$. The strong maximum principle then yields $h<0$ at inner points, since certainly $h\not\equiv 0$. Since $U_r = rW_{rr}$ it follows that
\[0> v_t = U_r\frac{\dd r}{\dd t}  = - r^2W_{rr}\]
and $W_{rr}>0$ in $D_1$. Integrating back, we see that also $W_r$ and $W$ are positive in $D_1$. Moreover, $U_{\theta\theta} = v_{\theta\theta} = v_t < 0$ and
\[\theta\mapsto U_\theta(r,\theta) = rW_{r\theta}(r,\theta) - W_\theta(r,\theta) = \frac{8}{\pi}\sum_{n=1}^\infty r^{m_n^2}\cos(m_n\theta)\]
is therefore strictly decreasing for each $0<r<1$. Its zero is obviously at $\theta = \pi/4$.
The following Lemma is proved.
\begin{lemma}\label{lem:1}
Let $W\colon\overline{D}_1\to\mR$ be the series defined in \eqref{eq:Wdef}. 
\begin{enumerate}[(I)]
\item The functions $W$, $W_r$, and $W_{rr}$ are positive in $D_1$.
\item $W,W_\theta\in C\left(\overline{D}_1\right)$, and $W_r\in C\left(\overline{D}_1\setminus\{(1,0),(1,\pi/2)\}\right)$. $W_r$ is discontinuous at the two points $(r,\theta) = (1,0)$ and $(r,\theta) = (1,\pi/2)$.
\item For $0<r<1$,
\[U_\theta(r,\theta) = rW_{r\theta}(r,\theta) - W_\theta(r,\theta)\]
is positive when $0\leq\theta<\pi/4$, negative when $\pi/4<\theta\leq\pi/2$, and thus zero in $D_1$ precisely when $\theta=\pi/4$.
\end{enumerate}
\end{lemma}
%

Let
\[B_1 := \left\{(p,q)\,|\, 0<p^2 + q^2<1,\; p>0,\,q>0\right\}\]
be the sector in the first quadrant and define $w\colon \oB_1\to\mR$ as $W$ in Cartesian coordinates, i.e., $w(r\cos\theta,r\sin\theta) := W(r,\theta)$.
I shift to vector notation
\[\bfx := [x,y]^\top\in\oO_1,\, \bfp := [p,q]^\top\in\oB_1,\, \bfr := [r,\theta]^\top\in\overline{D}_1,\]
and denote by $\Phi$ the coordinate transformation $\Phi\colon\overline{D}_1\to\oB_1$. i.e.,
\[\Phi(\bfr) = \Phi(r,\theta) := \begin{bmatrix}
r\cos\theta\\
r\sin\theta
\end{bmatrix} = \begin{bmatrix}
p\\
q
\end{bmatrix} = \bfp.\]
Now,
\[W(\bfr) = w(\Phi(\bfr))\]
and $\nabla W(\bfr) = \nabla w(\Phi(\bfr))\nabla\Phi(\bfr)$ where $\nabla W = [W_r, W_\theta]$, $\nabla w = [w_p, w_q]$, and
\[\nabla\Phi = \begin{bmatrix}
\nabla (r\cos\theta)\\
\nabla(r\sin\theta)
\end{bmatrix} =
\begin{bmatrix}
\cos\theta & -r\sin\theta\\
\sin\theta & r\cos\theta
\end{bmatrix}\]
is the Jacobian matrix of $\Phi$.

In the next Lemma, I establish that
\emph{for every $\bfx = [x,y]^\top\in\Omega_1$ the minimax in \eqref{eq:solution} is obtained at a unique point $\bfr = [r,\theta]^\top\in D_1$.}
The result is crucial because it allows us to set up a correspondence $\bfg\colon\Omega_1\to B_1$ as
\begin{equation}\label{eq:g_def}
\bfg(\bfx) = \bfg(x,y) := \begin{bmatrix}
r\cos\theta\\
r\sin\theta
\end{bmatrix} = \Phi(\bfr),\qquad \text{$[r,\theta]^\top\in D_1$ is \emph{the} minimax in \eqref{eq:solution}.}
\end{equation}

\begin{lemma}\label{lem:interior_minmax}
Fix $[x,y]^\top\in\Omega_1$. The gradient of the objective function
\[\begin{bmatrix}
r\\
\theta
\end{bmatrix}\mapsto f_\bfx(r,\theta)
	= r(x\cos\theta + y\sin\theta) - W(r,\theta)\]
	is zero at exactly one point in $\overline{D}_1$. This critical point is an \emph{interior} minimax.
\end{lemma}
Of course, the critical \emph{value} is $u(x,y)$ as defined in \eqref{eq:solution}. 

\begin{proof}
Let $\bfx = [x,y]^\top\in\Omega_1$ and assume first that $\theta\in(0,\pi/2)$. Then the partial derivative of the objective function
\[r\mapsto\frac{\partial}{\partial r}f_\bfx(r,\theta) = x\cos\theta + y\sin\theta - W_r(r,\theta)\]
is continuous up to the boundary (Lemma \ref{lem:1} (II))
with end-point values
	\[\frac{\partial}{\partial r}f_\bfx(0,\theta) = x\cos\theta + y\sin\theta - 0 > 0\]
	and
	\begin{align*}
	\frac{\partial}{\partial r}f_\bfx(1,\theta) &= x\cos\theta + y\sin\theta - W_{r}(1,\theta)\\
	&= x\cos\theta + y\sin\theta + W_{\theta\theta}(1,\theta)\\
	&= x\cos\theta + y\sin\theta - \cos\theta - \sin\theta\\
	&= -(1-x)\cos\theta - (1-y)\sin\theta < 0.
	\end{align*}
	Thus, $f_\bfx(\cdot,\theta)$ has an interior maximum for each fixed $\theta\in(0,\pi/2)$. It is unique since
	\[\frac{\partial^2}{\partial r^2}f_\bfx(r,\theta) = - W_{rr}(r,\theta)<0\]
	by Lemma \ref{lem:1} (I).
	It occurs at $r$ where
	\[0 = \frac{\partial}{\partial r}f_\bfx(r,\theta) = x\cos\theta + y\sin\theta - W_{r}(r,\theta)\]
	and by the implicit function theorem there is an analytic function $r_\bfx\colon(0,\pi/2)\to(0,1)$ so that
	\begin{equation}\label{eq:rx_relation}
	W_r(r_\bfx(\theta),\theta) = x\cos\theta + y\sin\theta.
	\end{equation}
	
	At $\theta = 0$ and $\theta = \pi/2$ we have $f_\bfx(r,0) = rx$ and $f_\bfx(r,\pi/2) = ry$ with maximum value $x$ and $y$, respectively, at $r = 1$. We thus extend the function $r_\bfx$ to the closed interval $[0,\pi/2]$ by defining $r_\bfx(0) = r_\bfx(\pi/2) = 1$.
	
	Observe that the limit
\[\lim_{\theta\to 0^+}W_r(r_\bfx(\theta),\theta) = \lim_{\theta\to 0^+}(x\cos\theta + y\sin\theta) = x>0\]
	exists. Since $W_r(r,0) = 0$ for all $0\leq r \leq 1$, the only possibility is that $r_\bfx(\theta)\to 1$. Recall that $W_r$ is not continuous at $r=1$, $\theta=0$ (Lemma \ref{lem:1} (II)). Similarly, $r_\bfx(\theta)\to 1$ also when $\theta\to\pi/2^-$ and we conclude that $r_\bfx$ is continuous up to the boundary.
	
	We can now write
	\[h_\bfx(\theta) := \max_{r\in[0,1]}f_\bfx(r,\theta) = 
	r_\bfx(\theta)(x\cos\theta + y\sin\theta) - W(r_\bfx(\theta),\theta),\quad \theta\in[0,\pi/2],
	\]
and any critical point of $f_\bfx$ must occur at $[r_\bfx(\theta),\theta]^\top$ for some $\theta\in[0,\pi/2]$.
	
	Note again that $h_\bfx$ is continuous and analytic in the interior.
	Differentiating yields
	\begin{align*}
	h_\bfx'(\theta) &= r_\bfx'(\theta)(x\cos\theta + y\sin\theta) + r_\bfx(\theta)(-x\sin\theta + y\cos\theta)\\
	&\quad{}- r_\bfx'(\theta)W_r(r_\bfx(\theta),\theta) - W_\theta(r_\bfx(\theta),\theta)\\
	&= r_\bfx(\theta)(-x\sin\theta + y\cos\theta) - W_\theta(r_\bfx(\theta),\theta)
	\end{align*}
	where the first and third terms disappeared by \eqref{eq:rx_relation}. Since $W_\theta(1,\theta) = -\sin\theta + \cos\theta$, the end-point values are $h_\bfx'(0) = y-1 < 0$ and $h_\bfx'(\pi/2) = -x + 1 > 0$, so $h_\bfx$ must have a minimum in the interior of the interval $[0,\pi/2]$.
	Next,
\begin{align*}
h_\bfx''(\theta)
	&= r_\bfx'(\theta)(-x\sin\theta + y\cos\theta) - r_\bfx(\theta)(x\cos\theta + y\sin\theta)\\
	&\quad{} - r_\bfx'(\theta)W_{r\theta}(r_\bfx(\theta),\theta)- W_{\theta\theta}(r_\bfx(\theta),\theta)
\end{align*}
and the second and fourth terms cancel because of \eqref{eq:rx_relation} and $-W_{\theta\theta} = rW_r$. Moreover, differentiating \eqref{eq:rx_relation} and rearranging yields
	\[- x\sin\theta + y\cos\theta - W_{r\theta}(r_\bfx(\theta),\theta) = r_\bfx'(\theta)W_{rr}(r_\bfx(\theta),\theta)\]
	and thus,
	\[h_\bfx''(\theta) = (r_\bfx'(\theta))^2W_{rr}(r_\bfx(\theta),\theta) \geq 0.\]
	This means that $h_\bfx'$ is strictly increasing since $W_{rr}>0$ and since $(r_\bfx')^2$ is analytic and can therefore not be zero over an interval of positive length. It follows that the minimum of $h_\bfx$ is unique and the proof is concluded.
\end{proof}


The explicit formulas for the gradient and Hessian matrix of $w(\bfp) = W(\Phi^{-1}(\bfp))$ are now needed.

\begin{lemma}\label{lem:Hw}
	In terms of $W$, the gradient $\nabla w = [w_p, w_q]$ and the Hessian matrix $\cH w = \begin{bmatrix}
	w_{pp} & w_{pq}\\
	w_{pq} & w_{qq}
	\end{bmatrix}$ of $w$ at $\bfp\in B_1$ are
	\[\nabla w(\bfp) = \left[\cos\theta\, W_r - \frac{1}{r}\sin\theta\, W_\theta,\;\sin\theta\, W_r + \frac{1}{r}\cos\theta\, W_\theta\right]\]
	and
	\[\cH w(\bfp)
	= (\nabla\Phi)^{-\top}\begin{bmatrix}
	W_{rr} & W_{r\theta} - \frac{1}{r}W_\theta\\
	W_{r\theta} - \frac{1}{r}W_\theta & 0
	\end{bmatrix}(\nabla\Phi)^{-1}
	\]
	at $\bfr = \Phi^{-1}(\bfp)$. 
\end{lemma}

Recall, $\Phi(\bfr) = \begin{bmatrix}
r\cos\theta\\
r\sin\theta
\end{bmatrix}$, $\nabla\Phi(\bfr) = \begin{bmatrix}
\cos\theta & -r\sin\theta\\
\sin\theta & r\cos\theta
\end{bmatrix}$
and thus
\[(\nabla\Phi)^{-1}(\bfr) = \frac{1}{r}\begin{bmatrix}
r\cos\theta & r\sin\theta\\
-\sin\theta & \cos\theta
\end{bmatrix}.\]

By writing out the product one can check that
\begin{equation}\label{eq:sec5_fullHw}
\cH w(\bfp)
= \frac{1}{r^2}\begin{bmatrix}
c^2 r^2W_{rr} - 2sc U_\theta & sc r^2W_{rr} + (c^2-s^2)U_\theta\\
sc r^2W_{rr} + (c^2-s^2)U_\theta & s^2 r^2W_{rr} + 2sc U_\theta
\end{bmatrix}
\end{equation}
where $U(r,\theta) = rW_r(r,\theta) - W(r,\theta)$, and $s = \sin\theta$ and $c = \cos\theta$.

\begin{proof}
	Applying the chain rule to the relation $W(\bfr) = w(\Phi(\bfr))$ yields
	\begin{align*}
	\nabla w(\Phi(\bfr))
	&= \nabla W(\bfr)(\nabla\Phi)^{-1}(\bfr)\\
	&= \frac{1}{r}[W_r, W_\theta]\begin{bmatrix}
	r\cos\theta & r\sin\theta\\
	-\sin\theta & \cos\theta
	\end{bmatrix}\\
	&= \left[\cos\theta\, W_r - \frac{1}{r}\sin\theta\, W_\theta,\;\sin\theta\, W_r + \frac{1}{r}\cos\theta\, W_\theta\right].
	\end{align*}
	
	Next, $\nabla W^\top(\bfr) = \nabla\Phi^\top(\bfr)\nabla w^\top(\Phi(\bfr))$ and
	$\cH W = \nabla\Phi^\top\cH w\nabla \Phi + \nabla_{\nabla w}\nabla\Phi^\top$. Thus,
	\begin{equation}\label{eq:Hw}
	\cH w = (\nabla\Phi)^{-\top}\left(\begin{bmatrix}
	W_{rr} & W_{r\theta}\\
	W_{r\theta} & W_{\theta\theta}
	\end{bmatrix} - \nabla_{\nabla w}\nabla\Phi^\top\right)(\nabla\Phi)^{-1}.
	\end{equation}
	The notation $\nabla_{\nabla w}\nabla\Phi^\top$ is short-hand for the Jacobian matrix of the vector field $\bfr\mapsto \nabla\Phi^\top(\bfr)\bfa$ ($\bfa\in\mR^2$ constant) evaluated at $\bfa = \nabla w^\top(\Phi(\bfr))$. Since
	\begin{align*}
	\nabla\left(\nabla\Phi^\top\bfa\right)
	&= \nabla\left(\begin{bmatrix}
	\cos\theta & \sin\theta\\
	-r\sin\theta & r\cos\theta
	\end{bmatrix}\begin{bmatrix}
	a\\ b
	\end{bmatrix}\right)\\
	&= \nabla\left(a\begin{bmatrix}
	\cos\theta\\ -r\sin\theta
	\end{bmatrix} + b\begin{bmatrix}
	\sin\theta\\ r\cos\theta
	\end{bmatrix}\right)\\
	&= a\begin{bmatrix}
	0 & -\sin\theta\\ -\sin\theta & -r\cos\theta 
	\end{bmatrix} + b\begin{bmatrix}
	0 & \cos\theta\\ \cos\theta & -r\sin\theta
	\end{bmatrix}
	\end{align*}
	we get
	\begin{align*}
	\nabla_{\nabla w}\nabla\Phi^\top
	&= \frac{1}{r}\begin{bmatrix}
	0 & -r\sin\theta\cos\theta W_r + \sin^2\theta W_\theta\\
	-r\sin\theta\cos\theta W_r + \sin^2\theta W_\theta & -r^2\cos^2\theta W_r + \sin\theta\cos\theta W_\theta
	\end{bmatrix}\\
	&\quad{} + \frac{1}{r}\begin{bmatrix}
	0 & r\sin\theta\cos\theta W_r + \cos^2\theta W_\theta\\
	r\sin\theta\cos\theta W_r + \cos^2\theta W_\theta & -r^2\sin^2\theta W_r - \sin\theta\cos\theta W_\theta
	\end{bmatrix}\\
	&= \begin{bmatrix}
	0 & \frac{1}{r}W_\theta\\
	\frac{1}{r}W_\theta & -r W_r
	\end{bmatrix}.
	\end{align*}
	Plugging this into \eqref{eq:Hw} gives the result as the lower right entry becomes $W_{\theta\theta} + rW_r = 0$.
\end{proof}

\begin{lemma}\label{lem:Dw}
The gradient
	\[\nabla w(\bfp) = \left[\cos\theta\, W_r - \frac{1}{r}\sin\theta\, W_\theta,\;\sin\theta\, W_r + \frac{1}{r}\cos\theta\, W_\theta\right]\]
is continuous up to the boundary $\partial B_1$ except at the points $[p,q] = [1,0]$ and $[p,q] = [0,1]$.

Moreover, $\nabla w^\top(B_1) \subseteq \Omega_1$, $\nabla w^\top(\partial B_1) \subseteq \partial\Omega_1$, and if the limit
\begin{equation}\label{eq:sec5boundary_map}
\bfx_* := \lim_{j\to \infty}\nabla w^\top(\bfp_j)
\end{equation}
exists for some sequence $\bfp_j\in B_1$ converging to a boundary point $\bfp_*\in\partial B_1$, then $\bfx_*\in\partial\Omega_1$.
\end{lemma}

\begin{proof}
The boundary continuity follows from Lemma \ref{lem:1} (II). I calculate the boundary values of $w_q(p,q) = w_q(r\cos\theta,r\sin\theta)$. First, with $p=0$ and $0\leq q \leq 1 $, we have $\theta = \pi/2$ and
\[w_q(0,q) = W_r(q,\pi/2) - 0 = 0.\]
For $p^2 + q^2 = 1$ and $0<p\leq 1$ we have $r=1$, $0\leq\theta<\pi/2$ and
	\begin{align*}
	w_q(\cos\theta,\sin\theta)
	&= \sin\theta W_r(1,\theta) + \cos\theta W_\theta(1,\theta)\\
	&= \sin\theta(\cos\theta + \sin\theta) + \cos\theta (\cos\theta - \sin\theta)\\
	&= 1.
	\end{align*}
	Indeed, the factor $\sin\theta$ cancels the discontinuity of $W_r(1,\cdot)$ at $\theta = 0$.
	Finally, for $q=0$ we have $\theta = 0$ and
	\[w_q(p,0) = \frac{1}{p}W_\theta(p,0).\]
	This function is continuous, starts in $w_q(0,0) = 0$ and ends in $w_q(1,0) = 1$. It is monotone since
\begin{align*}
p^2w_{pq}(p,0)
	&= pW_{r\theta}(p,0) - W_{\theta}(p,0)\\
	&= U_\theta(p,0) > 0
\end{align*}
by Lemma \ref{lem:1} (III).	
The symmetry $w(p,q) = w(q,p)$ yields $w_p(p,q) = w_q(q,p)$ and the boundary values of $\nabla w$ are
\begin{align*}
\nabla w(0,q) &= \left[\frac{1}{q}W_\theta(q,0),\; 0\right], && 0\leq q \leq 1,\\
\nabla w(\sin\theta,\cos\theta) &= \left[1,1\right], && 0<\theta<\pi/2,\\
\nabla w(p,0) &= \left[0,\;\frac{1}{p}W_\theta(p,0)\right], && 0\leq p \leq 1.
\end{align*}
Notice the jumps at $[0,1]^\top$ and $[1,0]^\top$, and that $\nabla w^\top(\partial B_1) \neq \partial\Omega_1$. The medians are missing. Nevertheless, $\nabla w^\top(\partial B_1) \subseteq \partial\Omega_1$.
	
Next, I show that $\nabla w^\top(B_1) \subseteq \Omega_1$. Let $[p_0,q_0]^\top\in B_1$. From the boundary values for $w_q$ computed above, the function $v(p) := w_q(p,q_0)$ is continuous, starts in 0 at $p=0<p_0$ and ends in 1 at $p = \sqrt{1-q_0^2}>p_0$. 
The formula \eqref{eq:sec5_fullHw} for the Hessian matrix $\cH w$ yields
\[v'(p) = w_{pq} = \sin\theta\cos\theta\, W_{rr} + \frac{\cos^2\theta - \sin^2\theta}{r^2}\,U_\theta\]
where $U_\theta = rW_{r\theta} - W_\theta$. By Lemma \ref{lem:1} (I) the first term is positive and by part (III) in the same Lemma, $U_\theta$ has the same sign as the factor $\cos^2\theta - \sin^2\theta = \cos(2\theta)$. Thus, $w_{pq}>0$ in $B_1$ and the function $v$ is strictly increasing. This means that $0<w_q(p_0,q_0)<1$, and by the symmetry of $w$, the same holds for $w_p$. It follows that $\nabla w^\top(p_0,q_0)\in\Omega_1$.
	
Finally, let $\bfp_j\to\bfp_*\in\partial B_1$. Unless $\bfp_* = [1,0]^\top$ or $\bfp_* = [0,1]^\top$ it is clear that $\lim_{j\to \infty}\nabla w^\top(\bfp_j) \in \partial\Omega_1$. However, $w_q$ is continuous at $\bfp_* = [1,0]^\top$ with $w_q(1,0) = 1$. Since $0< w_p< 1$ it follows that if the limit
\[\lim_{j\to \infty}\nabla w(\bfp_j) = \lim_{j\to \infty}[w_p(\bfp_j),w_q(\bfp_j)] = \bfx_*^\top = [x_*, y_*]\]
exists, then $y_* = 1$ and $0\leq x_*\leq 1$. That is, $\bfx_*\in\partial\Omega_1$.

Again, a symmetric argument holds for the case $\bfp_* = [0,1]^\top$.
\end{proof}

\begin{proposition}\label{prop:alternative_formula}
For interior points $\bfx\in\Omega_1$, the function \eqref{eq:solution} has the alternative formula
\begin{equation}\label{eq:alt_formula}
u(\bfx) = \bfx^\top\bfg(\bfx) - w(\bfg(\bfx)).
\end{equation}
Moreover, $\bfg\colon \Omega_1\to B_1$ is continuous and is the inverse of $\nabla w^\top\colon B_1\to \Omega_1$.
\end{proposition}

\begin{proof}
The formula for $u$ follows directly from the definition \eqref{eq:g_def} of $\bfg$:
\begin{align*}
u(\bfx) &= \min_{\theta\in[0,\pi/2]}\max_{r\in[0,1]}\left\{\bfx^\top\Phi(\bfr) - W(\bfr)\right\}\\
&= \min_{\theta\in[0,\pi/2]}\max_{r\in[0,1]}\left\{\bfx^\top\Phi(\bfr) - w(\Phi(\bfr))\right\}\\
&= \bfx^\top\bfg(\bfx) - w(\bfg(\bfx)),\qquad \bfx\in\Omega_1.
\end{align*}

By Lemma \ref{lem:interior_minmax}, the minimax of the objective function
\[f_\bfx(\bfr) = \bfx^\top\Phi(\bfr) - W(\bfr)\]
is attained at an interior critical point.
That is,
\begin{align*}
0 = \nabla f_\bfx(\bfr) &= \bfx^\top\nabla\Phi(\bfr) - \nabla W(\bfr)\\
  &= \left(\bfx^\top - \nabla w(\Phi(\bfr))\right)\nabla\Phi(\bfr)
\end{align*}
at $\Phi(\bfr) = \bfg(\bfx)$.
Thus,
\begin{equation}\label{eq:x_eq_nabla_w_g_x}
\bfx = \nabla w^\top(\bfg(\bfx))\qquad\text{when $\bfx\in\Omega_1.$}
\end{equation}
Next, let $\bfp\in B_1$. Then $\nabla w^\top(\bfp)\in\Omega_1$ by Lemma \ref{lem:Dw}, and from Lemma \ref{lem:interior_minmax} the solution $\bfr = \Phi^{-1}(\bfp)$ of the equation
\[0 = \nabla f_{\nabla w^\top(\bfp)}(\bfr) = \left(\nabla w(\bfp) - \nabla w(\Phi(\bfr))\right)\nabla\Phi(\bfr)\]
is unique. Since $\bfg(\nabla w^\top(\bfp)) := \Phi(\bfr)$, it follows that
\[\bfp = \bfg(\nabla w^\top(\bfp))\qquad \text{when $\bfp\in B_1.$}\]

It remains to confirm that $\bfg$ is continuous. To that end, let $\bfx_*\in\Omega_1$ and let $\bfx_j$ be an arbitrary sequence in $\Omega_1$ converging to $\bfx_*$ as $j\to\infty$. Define the sequence $\bfp_j := \bfg(\bfx_j)$ in $B_1$. By compactness there is a subsequence $\bfp_{j_i}$ and a point $\bfp_*\in\oB_1$ such that $\lim_{i\to \infty}\bfp_{j_i} = \bfp_*$. However, by \eqref{eq:sec5boundary_map} in Lemma \ref{lem:Dw}, $\bfp_*$ must be an interior point since the limit
\[\lim_{i\to \infty}\nabla w^\top(\bfp_{j_i}) = \lim_{i\to \infty}\nabla w^\top(\bfg(\bfx_{j_i})) = \lim_{i\to \infty}\bfx_{j_i} = \bfx_* \in\Omega_1\]
exists.

Since $\nabla w$ is smooth in $B_1$
we also have
\[\lim_{i\to \infty}\nabla w^\top(\bfp_{j_i}) = \nabla w^\top(\bfp_*),\]
and the inverse relation \eqref{eq:x_eq_nabla_w_g_x} yields
\[\nabla w^\top(\bfg(\bfx_*)) = \bfx_* = \nabla w^\top(\bfp_*).\]
The injectivety of $\nabla w$ in $B_1$ then implies
\[\bfg(\bfx_*) = \bfp_* = \lim_{i\to \infty}\bfg(\bfx_{j_i})\]
where the last equality is just the definition of $\bfp_*$. Thus, for every sequence $\bfx_j\to\bfx_*$ there is a subsequence $\bfx_{j_i}$ such that $\bfg(\bfx_{j_i})\to\bfg(\bfx_*)$, and it is proved that $\bfg$ is continuous at $\bfx_*$.
\end{proof}

Denote the diagonal in $\oO_1$ by $\delta := \{[x,x]^\top\,|\, 0\leq x\leq 1\}$.

\begin{proposition}\label{prop:realanal}
The function $\bfg$ is real-analytic in (each connected component of) $\Omega_1\setminus\delta$. So is $u(\bfx) = \bfx^\top\bfg(\bfx) - w(\bfg(\bfx))$.
\end{proposition}

\begin{proof}
It was proved in Proposition \ref{prop:alternative_formula} that $\bfg\colon\Omega_1\to B_1$ is the inverse of the analytic function $\nabla w^\top\colon B_1\to\Omega_1$. By the inverse function theorem, $\bfg$ is analytic at $\bfx$ provided the Hessian matrix $\cH w$ of $w$ is non-singular at $\bfp = \bfg(\bfx)$.
By Lemma \ref{lem:Hw},
\[\cH w(\bfp)
= (\nabla\Phi)^{-\top}\begin{bmatrix}
W_{rr} & W_{r\theta} - \frac{1}{r}W_\theta\\
W_{r\theta} - \frac{1}{r}W_\theta & 0
\end{bmatrix}(\nabla\Phi)^{-1}
\]
at $\bfr = \Phi^{-1}(\bfp)$.
The determinant of $\cH w$ is then
\begin{equation}\label{eq:Hw_det}
\det\cH w = -r^{-2}\left(W_{r\theta} - \frac{1}{r}W_\theta\right)^2 \leq 0,
\end{equation}
which by Lemma \ref{lem:1} (III) is zero in $D_1$ only when $\theta = \pi/4$.
Next, one may easily check that $W_\theta(r,\pi/4)\equiv 0$, and if the minimax is at $(r_0,\pi/4)$ then
\begin{align*}
0 &= \frac{\partial}{\partial\theta}f_{(x,y)}(r_0,\theta)\Big|_{\theta=\pi/4}\\
  &= r_0(-x\sin\theta + y\cos\theta) - W_\theta(r_0,\theta)\Big|_{\theta=\pi/4}\\
  &= \frac{r_0}{\sqrt{2}}(y-x) - 0
\end{align*}
and $x = y$ as $r_0>0$. It follows that $x\neq y$ implies $\theta\neq\pi/4$ and hence $\bfg$ is real-analytic in $\Omega_1\setminus\delta$.
\end{proof}

It is easily seen that $rW_r + W_{\theta\theta} = 0$ in the $r,\theta$-plane. The corresponding equation in the $p,q$-variables is
\begin{equation}\label{eq:weq}
p^2w_{qq} - 2pqw_{pq} + q^2w_{pp} = 0.
\end{equation}

\begin{proposition}\label{prop:inftyharmonic_smooth}
The function \eqref{eq:solution} is $\infty$-harmonic in the smooth sense away from the diagonal. i.e.,
\[\nabla u(\bfx)\cH u(\bfx)\nabla u^\top(\bfx) = 0\qquad\text{for all $\bfx\in\Omega_1\setminus\delta$.}\]
\end{proposition}

\begin{proof}
By Proposition \ref{prop:alternative_formula} we have that $\bfx = \nabla w^\top(\bfg(\bfx))$ and $u(\bfx) = \bfx^\top\bfg(\bfx) - w(\bfg(\bfx))$, and everything is smooth in $\Omega_1\setminus\delta$ by Proposition \ref{prop:realanal}. Differentiating these two identities yields,
\[I = \cH w(\bfg(\bfx))\nabla\bfg(\bfx)\]
and
\[\nabla u(\bfx) = \bfg^\top(\bfx) + \bfx^\top\nabla\bfg(\bfx) - \nabla w(\bfg(\bfx))\nabla\bfg(\bfx) = \bfg^\top(\bfx).\]
Thus, $\cH u(\bfx) = \nabla\bfg(\bfx) = (\cH w)^{-1}(\bfg(\bfx))$ -- the inverse of the Hessian matrix of $w$ at $\bfg(\bfx)$. We substitute $[p,q]^\top := \bfg(\bfx)$ and compute
\begin{align*}
\Delta_\infty u(\bfx)
	&= \nabla u(\bfx)\cH u(\bfx)\nabla u^\top(\bfx)\\
	&= \bfg^\top(\bfx)(\cH w)^{-1}(\bfg(\bfx))\bfg(\bfx)\\
	&= \frac{1}{\det\cH w}[p,q]\begin{bmatrix}
	w_{qq} & -w_{pq}\\
	-w_{pq} & w_{pp}
	\end{bmatrix}\begin{bmatrix}
	p\\ q
	\end{bmatrix}\\
	&= \frac{p^2w_{qq} - 2pqw_{pq} + q^2w_{pp}}{\det\cH w} = 0
\end{align*}
by \eqref{eq:weq} and because $\det\cH w\neq 0$ at $\bfg(\bfx)$ when $\bfx\notin\delta$.
\end{proof}

It is convenient to introduce the notation
\[\bben := \frac{1}{\sqrt{2}}\begin{bmatrix}
1\\ 1
\end{bmatrix},\qquad \bben_\perp := \frac{1}{\sqrt{2}}\begin{bmatrix}
-1\\ 1
\end{bmatrix}.\]
The objective function satisfies
\begin{equation}\label{eq:fsymmetry}
f_{(x,y)}(r,\pi/2-\theta) = f_{(y,x)}(r,\theta).
\end{equation}
It is just a direct computation.
This symmetry implies that the obtained minimax must have $\theta$-coordinate $\pi/4$ when $x=y$.
The converse statement was derived in the proof of Proposition \ref{prop:realanal}. In terms of $\bfg$, the property can be summarized as
\begin{equation}
\text{$\bfg(\bfx)$ is parallel to $\bben$}\qquad\iff\qquad\bfx\in\delta.
\end{equation}

\begin{proposition}\label{prop:diagonal_value}
For $s\in(0,\sqrt{2})$ we have
\[u(s\bben) = sg(s) - W(g(s),\pi/4)\]
where $g\colon(0,\sqrt{2})\to\mR$
is the inverse of the function
\begin{align*}
r\mapsto W_r(r,\pi/4)
	&= \frac{8}{\pi}\sum_{n=1}^\infty (-1)^{n-1}\frac{m_n}{m_n^2-1}r^{m_n^2-1},\qquad m_n := 4n-2,\\
	&= \frac{8}{\pi}\left(\frac{2}{3}r^3 - \frac{6}{35}r^{35} + \frac{10}{99}r^{99} - \cdots\right).
\end{align*}
Moreover,
\[\frac{\dd}{\dd s}u(s\bben) = g(s) = \bfg^\top(s\bben)\bben = |\bfg(s\bben)|\]
and $s\mapsto	u(s\bben)$ is analytic, strictly increasing, and convex.
\end{proposition}

\begin{proof}
Since $\bfg(s\bben)$ is parallel to $\bben$, it must be on the form $\bfg(s\bben) = g(s)\bben$ for some continuous scalar function $g$. The polar coordinates for $\bfg$ becomes $r=g(s)$, $\theta = \pi/4$, and the formula for $u$ on the diagonal then follow from Proposition \ref{prop:alternative_formula}.

As $\bfg^\top$ is the inverse of $\nabla w$, we have
\[s\bben^\top = \nabla w(\bfg(s\bben)) = \nabla w(g(s)\bben) = \nabla W(g(s),\pi/4)(\nabla\Phi)^{-1}(g(s),\pi/4),\]
which can be computed to $W_r(g(s),\pi/4)\bben^\top$. See Lemma \ref{lem:Hw}. That is, $g$ is the inverse of $W_r(r,\pi/4)$. The series is alternating because $\sin (m_n\pi/4) = (-1)^{n-1}$.
	
Note that $g$, and thus also $s\mapsto u(s\bben)$, is analytic since $W_{rr}>0$. By a direct differentiation, $\frac{\dd}{\dd s}u(s\bben) = g(s) = r > 0$ and $u$ is strictly increasing along the diagonal. Finally, $s\mapsto u(s\bben)$ is convex since $\frac{\dd^2}{\dd s^2}u(s\bben) = g'(s) = 1/W_{rr}(g(s),\pi/4) > 0$.
\end{proof}

\begin{proposition}\label{prop:uC1}
The function \eqref{eq:solution} is $C^1$ in $\Omega_1$ with gradient
\[\nabla u(\bfx) = \bfg^\top(\bfx).\]
\end{proposition}

By Proposition \ref{prop:alternative_formula}, the function $\bfg$ is continuous in $\Omega_1$ and $u$ is given by $u(\bfx) = \bfx^\top\bfg(\bfx) - w(\bfg(\bfx))$.
However, (as we shall see) \emph{$\bfg$ is not differentiable over the diagonal} and we cannot differentiate the formula for $u$ directly to obtain $\nabla u = \bfg^\top$, as we did in Proposition \ref{prop:inftyharmonic_smooth}.

\begin{proof}
I aim to prove $u(\bfx_0 + \bfh) - u(\bfx_0) = \bfg^\top(\bfx_0)\bfh + o(|\bfh|)$ as $\bfh\to 0$, also for points $\bfx_0$ on the diagonal.

Write $\bfx_0 = s_0\bben\in\delta\cap\Omega_1$ and decompose $\bfh$ into $h_1\bben + h_2\bben_\perp$. Then
\begin{align*}
u(\bfx_0 + \bfh) - u(\bfx_0)
	&= u(s_0\bben + h_1\bben + h_2\bben_\perp) - u(s_0\bben)\\
	&= u(s_0\bben + h_1\bben + h_2\bben_\perp) - u(s_0\bben + h_1\bben)\\
	&\quad{}+ u(s_0\bben + h_1\bben) - u(s_0\bben).
\end{align*}
Since $u$ is continuous on the closed interval from $s_0\bben + h_1\bben$ to $s_0\bben + h_1\bben + h_2\bben_\perp$ and smooth in the interior, there is a $\gamma\in(0,1)$ so that the first difference above is
\[u((s_0+h_1)\bben + h_2\bben_\perp) - u((s_0+h_1)\bben) = \bfg^\top((s_0+h_1)\bben + \gamma h_2\bben_\perp)h_2\bben_\perp.\]
The expression is $o(|\bfh|)$ since $\bfg$ is continuous and perpendicular to $\bben_\perp$ on the diagonal.

By Proposition \ref{prop:diagonal_value}, the second difference is precisely
\begin{align*}
u((s_0+h_1)\bben) - u(s_0\bben)
	&= g(s_0)h_1 + o(|h_1|)\\
	&= \bfg^\top(s_0\bben)\bben h_1 + o(|h_1|)\\
	&= \bfg^\top(\bfx_0)\bfh + o(|\bfh|),
\end{align*}
which completes the proof.
\end{proof}

In the final Proposition I show that $u$ is not twice differentiable on the diagonal. In particular,
\begin{proposition}
Let $\bfx_0\in\delta\cap\Omega_1$. For $|s|$ small, define $c(s) := u(\bfx_0 + s\bben_\perp)$. Then
\[\lim_{s\to 0}c''(s) = -\infty.\]
\end{proposition}



\begin{proof}
For $s\neq 0$, $c$ is smooth with second derivative $c''(s) = \bben_\perp^\top\cH u(\bfx_0 + s\bben_\perp)\bben_\perp$
By the formula for $\cH w(\bfg(\bfx)) = (\cH u)^{-1}(\bfx)$ from Lemma \ref{lem:Hw} we get
\[\cH u(\bfx)
= -\frac{1}{\left(W_{r\theta} - \frac{1}{r}W_\theta\right)^2}\nabla\Phi\begin{bmatrix}
0 & \frac{1}{r}W_\theta - W_{r\theta}\\
\frac{1}{r}W_\theta - W_{r\theta} & W_{rr}
\end{bmatrix}\nabla\Phi^\top
\]
at $\bfr = \Phi^{-1}(\bfg(\bfx))$.
Next,
\[\nabla\Phi^\top\bben_\perp = \frac{1}{\sqrt{2}}\begin{bmatrix}
\sin\theta - \cos\theta\\ r(\cos\theta + \sin\theta)
\end{bmatrix},\]
so
\begin{align*}
c''(s) &= \bben_\perp^\top\cH u(\bfx_0 + s\bben_\perp)\bben_\perp\\
	&= -\frac{1}{\left(W_{r\theta} - \frac{1}{r}W_\theta\right)^2}\bben_\perp^\top\nabla\Phi\begin{bmatrix}
	0 & \frac{1}{r}W_\theta - W_{r\theta}\\
	\frac{1}{r}W_\theta - W_{r\theta} & W_{rr}
	\end{bmatrix}\nabla\Phi^\top\bben_\perp\\
	&= -\frac{2r(\sin^2\theta-\cos^2\theta)\left(W_{r\theta} - \frac{1}{r}W_\theta\right) + r^2(\cos\theta + \sin\theta)^2 W_{rr}}{2\left(W_{r\theta} - \frac{1}{r}W_\theta\right)^2}
\end{align*}
at $\bfr = \Phi^{-1}(\bfg(\bfx_0 + s\bben_\perp))$.
This concludes the proof since $\theta\to\pi/4$ as $s\to 0$. The numerator then goes to $0 + 2r_0^2W_{rr}(r_0,\pi/4) > 0$ (Lemma \ref{lem:1} (I)) while the denominator goes to 0 from the positive side (Lemma \ref{lem:1} (III)).
\end{proof}

By this last result it is clear that no $C^2$ test function can touch $u$ from below on the diagonal. In order to complete the proof of $u$ being a viscosity solution to the Dirichlet problem \eqref{eq:Dirproblem2}, it only remains to consider test functions touching $u$ on the diagonal from above.

If a $C^2$ test function $\phi$ touches $u$ from above at $\bfx_0 = s_0\bben\in\delta\cap\Omega_1$, then $\nabla\phi(\bfx_0) = \nabla u(\bfx_0) = g(s_0)\bben^\top$ since $u$ is $C^1$ and
\begin{align*}
\Delta_\infty\phi(\bfx_0)
	&= g^2(s_0)\lim_{\epsilon\to 0}\frac{\phi(\bfx_0 - \epsilon\bben) - 2\phi(\bfx_0) + \phi(\bfx_0 + \epsilon\bben)}{\epsilon^2}\\
	&\geq g^2(s_0)\lim_{\epsilon\to 0}\frac{u(\bfx_0 - \epsilon\bben) - 2u(\bfx_0) + u(\bfx_0 + \epsilon\bben)}{\epsilon^2}\\
	&= g^2(s_0)g'(s_0) > 0.
\end{align*}
Thus it is proved that $u$ is a viscosity solution of the Dirichlet problem \eqref{eq:Dirproblem2} in $\Omega_1$.

We now glue the translations and reflections of the formula \eqref{eq:solution} in $\oO_1$ (call it $u_1\colon\oO_1\to\mR$) together to make the solution $u\colon\oO\to\mR$ of the problem \eqref{eq:Dirprob1}. More specifically,
\begin{equation}\label{eq:full_solution}
u(x,y) = \begin{cases}
u_1(x,y),\qquad &\text{for $0\leq x\leq 1,\;0\leq y\leq 1$},\\
u_1(2-x,y), &\text{for $1\leq x\leq 2,\;0\leq y\leq 1$},\\
u_1(2-x,2-y), &\text{for $1\leq x\leq 2,\;1\leq y\leq 2$},\\
u_1(x,2-y), &\text{for $0\leq x\leq 1,\;1\leq y\leq 2$}.
\end{cases}
\end{equation}
The bounds
\[1 - \sqrt{(1-x)^2 + (1-y)^2} \leq u(x,y) \leq \dist( (x,y),\partial\Omega)\]
then holds for all $[x,y]^\top\in\oO$. In particular, $u\in C^1(\Omega\setminus\{[1,1]^\top\})$ since it is squeezed between smooth functions along the gluing edges, i.e., the medians in the square $\Omega$. Furthermore, either of the bounds are locally $\infty$-harmonic at the medians and any test function $\phi$ touching $u$ from above or below at these lines will have the correct sign of $\Delta_\infty\phi$ at the touching point.

The proof of Theorem \ref{thm} is completed.

\section{The first approximation recovers Aronsson's function}

If the series $W$ is truncated, one should expect to get an approximation of the solution $u$. The series converge very fast for small $r\geq0$, and $W$ is dominated by its first term
\[W(r,\theta) \approx \frac{4}{3\pi}r^4\sin(2\theta).\]
In Cartesian coordinates $[p,q] = [r\cos\theta,r\sin\theta]$ this is
\[w(p,q) \approx \frac{4}{3\pi}r^4\sin(2\theta) = \frac{8}{3\pi}r^2r^2\sin\theta \cos\theta = \frac{8}{3\pi}(p^2+q^2)pq\]
with gradient
\[\nabla w(p,q) \approx \frac{8}{3\pi}[3p^2q + q^3, p^3 + 3pq^2].\]
We want to find the inverse of this function. That is, to solve the system $\nabla w(p,q) = [x,y]$ for $p$ and $q$. Adding and subtracting yields
\begin{align*}
\frac{3\pi}{8}(y + x) &= p^3 + 3p^2q + 3pq^2 + q^3 = (p+q)^3,\\
\frac{3\pi}{8}(y - x) &= p^3 - 3p^2q + 3pq^2 - q^3 = (p-q)^3,
\end{align*}
so 
\begin{align*}
p &= \frac{1}{2}\left(c(x+y)^{1/3} + c(y-x)^{1/3}\right),\\
q &= \frac{1}{2}\left(c(x+y)^{1/3} - c(y-x)^{1/3}\right),
\end{align*}
where $c := (3\pi)^{1/3}/2$. This defines the function $\bfg(x,y) = [p,q]^\top$, and
\[\nabla u(x,y) = \bfg^\top(x,y) \approx \frac{c}{2}\left[(x+y)^{1/3} + (y-x)^{1/3},\, (x+y)^{1/3} - (y-x)^{1/3}\right],\]
which we recognise as the gradient of
\begin{equation}\label{eq:first_approx}
u(x,y) \approx \frac{3c}{8}\left((x+y)^{4/3} - (y-x)^{4/3}\right),\qquad 0\leq x\leq 1,\,0\leq y\leq 1.
\end{equation}
It is a rotation of Aronsson's function. It will be a very good approximation to the solution $u$ when $r$ -- that is, the length of the gradient of $u$ -- is small. Near the boundary point $(1,1)$ the length of $\nabla u$ tends to its maximal value 1. Even so, the formula \eqref{eq:first_approx} yields the acceptable estimate
\[1 = u(1,1) \approx \frac{3}{8}(6\pi)^{1/3} = 0.99800...\]

\section{Connections to the Jacobi Theta function}\label{sec:diagonal}

The \nth{2} Jacobi Theta function is
\[\vartheta_2(z,q) = 2\sum_{k=1}^\infty q^{(k-1/2)^2}\cos((2k-1)z),\quad 0\leq q<1.\]
Differentiating $U := rW_r - W$ given in \eqref{eq:Udef} with respect to $\theta$ yields
\begin{align*}
U_\theta(r,\theta)
	&= rW_{r\theta}(r,\theta) - W_\theta(r,\theta)\\
	&= \frac{8}{\pi}\sum_{n=1}^\infty r^{(4n-2)^2}\cos((4n-2)\theta)\\
	&= \frac{4}{\pi}\vartheta_2(2\theta,r^{16}).
\end{align*}
The expression for the determinant of $\cH u = (\cH w)^{-1}$ in the Introduction then follows from \eqref{eq:Hw_det}.
Next, $u$ expressed as a function of $\nabla u = [r\cos\theta, r\sin\theta] = \Phi^\top(\bfr)$ is precisely 
\begin{align*}
u(\nabla w(\Phi(\bfr)))
	&=\nabla w(\Phi(\bfr))\Phi(\bfr) - w(\Phi(\bfr)),\qquad\text{by \eqref{eq:alt_formula},}\\
	&= rW_r(r,\theta) - W(r,\theta)\\
	&= U(r,\theta).
\end{align*}
Thus
\[(t,\theta)\mapsto u\left((\nabla u)^{-1}(e^{-t}\cos\theta,e^{-t}\sin\theta)\right) = U(e^{-t},\theta)\]
solves the heat equation and
\[u = \frac{4}{\pi}\int_0^\theta \vartheta_2(2\psi,r^{16})\dd\psi.\]

On the diagonal, $u(s\bben) = g(s)s - W(g(s),\pi/4)$ where $g$ is the inverse of
\[r\mapsto W_r(r,\pi/4) = \frac{8}{\pi}\left(\frac{2}{3}r^3 - \frac{6}{35}r^{35} + \frac{10}{99}r^{99} - \cdots\right).\]
This was proved in Proposition \ref{prop:diagonal_value}.
In terms of $r = g(s) = |\nabla u(s\bben)|$ the formula $u = rW_r(r,\pi/4) - W(r,\pi/4)$
is obtained by writing $s = W_r(g(s),\pi/4)$.

The Theta function has the product representation
\[\vartheta_2(z,q) = 2q^{1/4}\cos z\prod_{k=1}^{\infty}(1-q^{2k})(1 + 2q^{2k}\cos(2z) + q^{4k}),\quad 0\leq q<1.\]
See \cite{MR0178117}, page 464 and 470. The $z$-derivative is
\begin{align*}
\frac{\partial}{\partial z}\vartheta_2(z,q)
&= -\sin z\left(1 + \cos^2 z \sum_{j=1}^\infty \frac{8q^{2j}}{1 + 2q^{2j}\cos(2z) + q^{4j}}\right)\\
&\quad{} \times 2q^{1/4}\prod_{k=1}^{\infty}(1-q^{2k})(1 + 2q^{2k}\cos(2z) + q^{4k}).
\end{align*}
At $\theta = \pi/4$ we have $z = 2\theta = \pi/2$. Thus, $\cos z = 0$, $\cos(2z) = -1$, $\sin z = 1$, and it follows that
\begin{align*}
U_r(r,\pi/4)
	&= -\frac{1}{r}U_{\theta\theta}(r,\pi/4)\\
	&= -\frac{4}{\pi r}\frac{\partial}{\partial\theta}\vartheta_2(2\theta,r^{16})\Big|_{\theta = \pi/4}\\
	&= \frac{8}{\pi}r^3 \prod_{k=1}^{\infty}(1-r^{32k})^3\qquad \text{for $0\leq r < 1$,}
\end{align*}
and thus
\begin{equation}\label{eq:Ur_lim}
\lim_{r\nearrow 1}U_r(r,\pi/4) = 0.
\end{equation}
This fact is not so easily derived from the series representation
\[U_r(r,\pi/4) = rW_{rr}(r,\pi/4) = \frac{8}{\pi}\left(2r^3 - 6r^{35} + 10r^{99} - \cdots\right).\]

\section{The $\infty$-Potential is not the $\infty$-Ground State}

The viscosity solution of the problem
\begin{equation}\label{eq:eigfuneq}
\begin{cases}
\max\left\{\Lambda - \frac{|\nabla v|}{v},\; \Delta_\infty v\right\} = 0, &\text{in $\Omega$,}\\
v>0,\quad v\in C(\oO),\quad v|_{\partial\Omega} = 0
\end{cases}
\end{equation}
is the $\infty$-Ground State. According to \cite{MR1716563} the only possible value of $\Lambda$ is
\[\Lambda = \frac{1}{\max_\bfx \dist(\bfx,\partial\Omega)},\]
which becomes $\Lambda = 1$ for the square $\Omega$.
In \cite{MR1886623} it was predicted that $v=u$.
In \cite{https://doi.org/10.48550/arxiv.2004.08127} the bound $\max_{\Omega}|v-u|\lesssim 10^{-3}$ was numerically obtained, supporting the conjecture. However, my explicit formula for $u$ on the diagonal shows that $u$ \emph{is not} $v$.

In order to show that \eqref{eq:solution} is not a solution of \eqref{eq:eigfuneq}, consider the difference
\[d(r) := u - |\nabla u| = U(r,\pi/4) - r\]
between $u$ and $|\nabla u|$ on the diagonal.
The difference is continuous up to the boundary with end-point values $d(0) = d(1) = 0$.
By \eqref{eq:Ur_lim}, $\lim_{r\nearrow 1}U_r(r,\pi/4) = 0$. See Figure \ref{fig:diagonal}\textsc{(a)}. Thus, $\lim_{r\nearrow 1}d'(r) = -1$ and $d$ must be positive for some $r_0<1$. That is, $|\nabla u|$ is less than $u$ at $s_0\bben\in\Omega_1$ where $s_0 = W_r(r_0,\pi/4)$, as shown in Figure \ref{fig:diagonal}\textsc{(b)}.
It follows that
\[1 - \frac{|\nabla u(s_0\bben)|}{u(s_0\bben)} > 0\]
and hence $u$ cannot be an $\infty$-eigenfunction.

\begin{figure}[h]
	\centering
	\begin{subfigure}[b]{0.45\textwidth}
		\centering
		\includegraphics{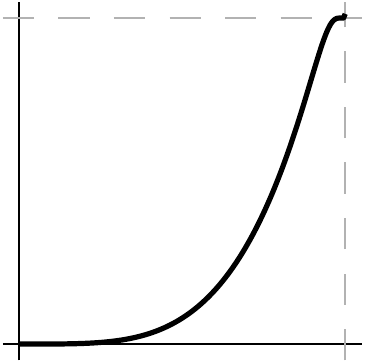}
		\caption{}
		\label{fig:u_and_r}
	\end{subfigure}
	\begin{subfigure}[b]{0.45\textwidth}
		\centering
		\includegraphics{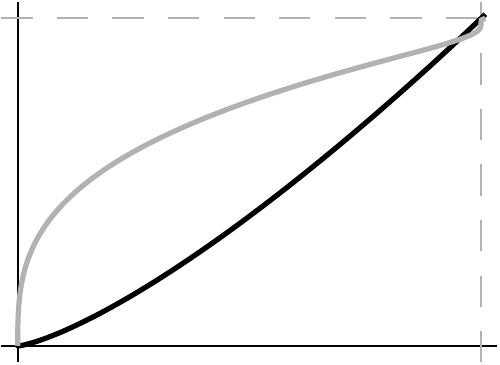}
		\caption{}
		\label{fig:bad}
	\end{subfigure}
	\caption{The diagonal: \textsc{(a)} $u$ as a function of
		$r=|\nabla u|\in[0,1]$. \textsc{(b)} $u(s\bben)$ and $|\nabla u(s\bben)|$ (gray) in arc-length parametrization $s\in[0,\sqrt{2}]$.}\label{fig:diagonal}
\end{figure}

%

\section{Deriving the solution}\label{sec:derive}

In this Section I shall give an heuristic explanation on how the formula \eqref{eq:solution} was obtained.
No rigour is needed and I put the emphasise on the method and the ideas.

I shall take advantage of some known properties of the solution to \eqref{eq:Dirprob1} in the square $\oO = [0,2]^2$. Firstly, by repeatedly making odd reflections along the edges, one can tessellate the plane by a function that is infinity-harmonic in $\mR^2$ except at coordinates $(m,n)$ where $m$ and $n$ are odd. Secondly, the bounds
\[1 - \sqrt{(1-x)^2 + (1-y)^2} \leq u(x,y) \leq \dist( (x,y),\partial\Omega),\]
which follow by the comparison principle, determines both $u$ and $\nabla u$ on the medians of $\oO$, and -- by extension -- on the grid lines in $\mR^2$ connecting the singularities $\{(m,n)\}$.  In particular, the gradient is known on the boundary of the square 
\[S := \{[x,y]^\top\;|\; -1<x<1,\,-1<y<1\},\]
which is a prerequisite in order to use the hodograph transform. 

I therefore consider the Dirichlet problem
\begin{equation}\label{eq:Dirprob_derive}
\begin{cases}
\Delta_\infty u = 0\qquad &\text{in $S$},\\
u(t,1) = u(1,t) = t,\\
u(t,-1) = u(-1,t) = -t, &\text{for $-1\leq t\leq 1$.}
\end{cases}
\end{equation}
The restriction $u|_{\oO_1}$ will then be the solution of \eqref{eq:Dirproblem2}.
This can also be seen by noting that the symmetry of the boundary conditions in \eqref{eq:Dirprob_derive} implies $u=0$ on the coordinate axis.

\begin{figure}[h]%
	\center
	\includegraphics{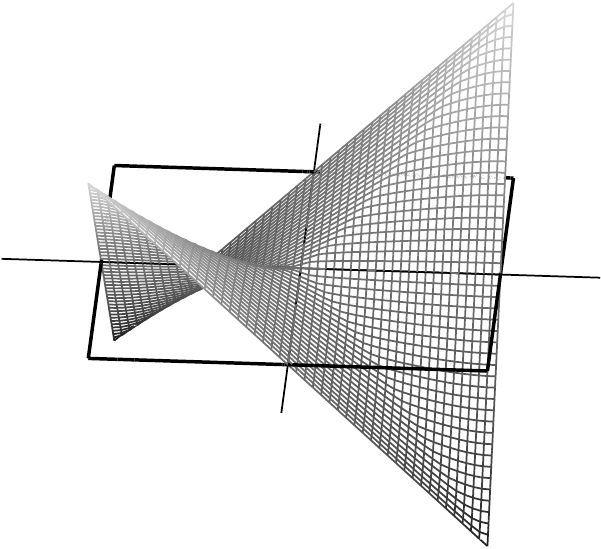}%
	\caption{The graph of the solution to \eqref{eq:Dirprob_derive} over the square $\overline{S}=[-1,1]^2$.}%
	\label{fig:derive_solution}%
\end{figure}

The following ansatz is made: The gradient $\nabla u\colon S\to B$ to the solution of \eqref{eq:Dirprob_derive} is one-to-one and onto
$B := \left\{(p,q)\in\mR^2\,|\, p^2 + q^2<1\right\}$ -- the unit disk.
Denote by $\bff\colon B\to S$ the inverse of $\nabla u$ and define $w\colon B\to\mR$ as
\[w(\bfp) := \bfp^\top\bff(\bfp) - u(\bff(\bfp)).\]
If $\bff$ is differentiable, then
\[\bfp = \nabla u^\top(\bff(\bfp))\qquad\text{implies}\qquad I = \cH u(\bff(\bfp))\nabla f(\bfp)\]
and
\begin{align*}
\nabla w(\bfp) &= \bff^\top(\bfp) + \bfp^\top\nabla\bff(\bfp) - \nabla u(\bff(\bfp))\nabla\bff(\bfp) = \bff^\top(\bfp),\\
\cH w(\bfp) &= \nabla\bff(\bfp) = (\cH u)^{-1}(\bff(\bfp)).
\end{align*}
Thus,
\begin{align*}
0 &= \nabla u(\bff(\bfp))\cH u(\bff(\bfp))\nabla u^\top(\bff(\bfp))\\
  &= \bfp^\top (\cH w)^{-1}(\bfp)\bfp\\
  &= [p,q]\begin{bmatrix}
	w_{pp} & w_{pq}\\
	w_{pq} & w_{qq}
  \end{bmatrix}^{-1}\begin{bmatrix}
  p\\ q
  \end{bmatrix}\\
  &= \frac{1}{w_{pp}w_{qq} - w_{pq}^2}[p,q]\begin{bmatrix}
  w_{qq} & -w_{pq}\\
  -w_{pq} & w_{pp}
  \end{bmatrix}\begin{bmatrix}
  p\\ q
  \end{bmatrix},
\end{align*}
which leads to the linear homogeneous equation
\begin{equation}\label{eq:derive_w}
0 = p^2w_{qq} - 2pqw_{pq} + q^2w_{pp}.
\end{equation}
The equation is degenerate elliptic, as can be seen from writing the above as
\[0 = [-q,p]\begin{bmatrix}
w_{pp} & w_{pq}\\
w_{pq} & w_{qq}
\end{bmatrix}\begin{bmatrix}
-q\\ p
\end{bmatrix} = \tr\left(A(\bfp)\cH w\right)\]
where $A\colon B\to\mathcal{S}^2_+$ is the one-rank matrix valued function
\begin{equation}\label{eq:def_A}
A(\bfp) := Q\bfp\bfp^\top Q^\top,\qquad Q := \begin{bmatrix}
0 & -1\\
1 & 0
\end{bmatrix}.
\end{equation}
Since the domain of $w$ is the unit disk, it is natural to introduce polar coordinates. Define
\[W(r,\theta) := w(r\cos\theta,r\sin\theta).\]
Then $W_\theta = - w_p r\sin\theta + w_q r\cos\theta$ and
\begin{align*}
W_{\theta\theta} &= - \left[-w_{pp} r\sin\theta + w_{pq} r\cos\theta\right]r\sin\theta\\
	&\quad{} -w_p r\cos\theta\\
	&\quad{} + \left[-w_{pq} r\sin\theta + w_{qq} r\cos\theta\right]r\cos\theta\\
	&\quad{} -w_q r\sin\theta\\
	&= q^2w_{pp} - 2pqw_{pq} + p^2w_{qq} - pw_p - qw_q.
\end{align*}
Since $rW_r = r(w_p\cos\theta + w_q\sin\theta) = pw_p + qw_q$ it follows that
\begin{equation}\label{eq:Weq}
rW_r + W_{\theta\theta} = 0.
\end{equation}

A separation of variables $W(r,\theta) = F(r)G(\theta)$ yields
\[\frac{rF'(r)}{F(r)} = n^2 = -\frac{G''(\theta)}{G(\theta)},\qquad n=0,1,2,\dots,\]
and suggests a solution on the form
\begin{equation}\label{eq:Wgeneral_form}
W(r,\theta) = \sum_{n=0}^\infty r^{n^2}\left(A_n\cos(n\theta) + B_n\sin(n\theta)\right).
\end{equation}

The boundary values of $w$ at $\partial B$ needs to be established. That is, the values $W(1,\theta)$ for $\theta\in[0,2\pi]$. Since $\nabla u = [1,0]$ at the upper boundary $\ell_u := \{[t,1]^\top\;|\; -1<t<1\}\subseteq \partial S$, the gradient $\nabla u$ is certainly not invertible in the closure of $S$. Nevertheless, allowing $\bff$ to be multivalued at $[1,0]^\top$ as $\bff(1,0) := \ell_u$, still produces -- with some goodwill -- the single value
\begin{align*}
w(1,0) &= [1,0]\bff(1,0) - u(\bff(1,0))\\
          &= [1,0]\begin{bmatrix}
			t\\ 1	
          \end{bmatrix} - u(t,1)\\
          &= t - t = 0
\end{align*}
of $w$ at $[1,0]^\top$. The same calculations apply for the remaining three edges, but this strategy settles the values $W(1,k\pi/2) = 0$, $k = 0,1,2,3$, only at four boundary points.

We must examine the behaviour of $\nabla u$ at the corners of $S$. 
When $\bfx\in S$ approaches the corner point $[1,1]^\top$ we know from Corollary 10 in \cite{1078-0947_2019_8_4731} that $|\nabla u(\bfx)|\to 1$. Also, the gradient is continuous, and since $\nabla u(\ell_u) = [1,0]$ on the upper boundary and $\nabla u(\ell_r) = [0,1]$ on the right boundary $\ell_r := \{[1,t]^\top\;|\; -1<t<1\}$, the gradient must take on all the 'intermediate' values $\nabla u = [\cos\theta,\sin\theta]$, $0\leq\theta\leq\pi/2$, in some limit $\bfx\to [1,1]^\top$. We therefore consider $\nabla u$ as multivalued at $[1,1]^\top$. The 'inverse' $\bff$ is then constant $\bff(\cos\theta,\sin\theta) = [1,1]^\top$ for $0\leq\theta\leq\pi/2$, and the boundary values for $W$ becomes
\begin{align*}
W(1,\theta) &= w(\cos\theta,\sin\theta)\\
	&= [\cos\theta,\sin\theta]\bff(\cos\theta,\sin\theta) - u(\bff(\cos\theta,\sin\theta))\\
	&= [\cos\theta,\sin\theta]\begin{bmatrix}
	1\\ 1
	\end{bmatrix} - u(1,1)\\
	&= \cos\theta + \sin\theta - 1,\qquad \theta\in[0,\pi/2].
\end{align*}

I now continue around the square $S$ in a clock-wise manner in order to derive the boundary values also for $\theta\in[\pi/2,2\pi]$.
At the lower edge $\ell_l := \{[t,-1]^\top\;|\; -1<t<1\}$ we have $\nabla u = [-1,0]$, so near the corner $[1,-1]^\top$ the gradient sweeps $[\cos\theta,\sin\theta]$, $\pi/2\leq\theta\leq\pi$. We define $\bff(\cos\theta,\sin\theta) := [1,-1]^\top$ for those values and get
\begin{align*}
W(1,\theta) &= w(\cos\theta,\sin\theta)\\
&= [\cos\theta,\sin\theta]\bff(\cos\theta,\sin\theta) - u(\bff(\cos\theta,\sin\theta))\\
&= [\cos\theta,\sin\theta]\begin{bmatrix}
1\\ -1
\end{bmatrix} - u(1,-1)\\
&= \cos\theta - \sin\theta + 1,\qquad \theta\in[\pi/2,\pi].
\end{align*}

Conducting the same analysis at the corner points $[-1,-1]^\top$ and $[-1,1]^\top$ will yield the values of $W$ all around the circle. Namely,
\begin{equation}
W(1,\theta) = \begin{cases}
\cos\theta + \sin\theta - 1,\qquad &\text{for $0\leq\theta\leq\pi/2$,}\\
\cos\theta - \sin\theta + 1,\qquad &\text{for $\pi/2\leq\theta\leq\pi$,}\\
-\cos\theta - \sin\theta - 1,\qquad &\text{for $\pi\leq\theta\leq 3\pi/2$,}\\
-\cos\theta + \sin\theta + 1,\qquad &\text{for $3\pi/2\leq\theta\leq 2\pi$.}
\end{cases}
\end{equation}

Note that $W(1,k\pi/2) = 0$ for integers $k$, as it should. The function is thus well defined and continuous. In fact, its derivative is
\[W_\theta(1,\theta) = |\cos\theta| - |\sin\theta|\]
for all $\theta$. As this function is continuous, even, and $\pi$-periodic, it follows that $W(1,\theta)$ is $C^1$, odd, and $\pi$-periodic.

The Fourier series of $|\cos\theta|$ and $|\sin\theta|$ can be calculated to
\[|\cos\theta| = \frac{2}{\pi} + \frac{4}{\pi}\sum_{n=1}^\infty\frac{(-1)^{n-1}}{4n^2 - 1}\cos(2n\theta),\quad |\sin\theta| = \frac{2}{\pi} - \frac{4}{\pi}\sum_{n=1}^\infty\frac{1}{4n^2 - 1}\cos(2n\theta),\]
so
\begin{align*}
W_\theta(1,\theta)
	&= |\cos\theta| - |\sin\theta|\\
	&= \frac{4}{\pi}\sum_{n=1}^\infty\frac{(-1)^{n-1} + 1}{4n^2 - 1}\cos(2n\theta)\\
	&= \frac{8}{\pi}\sum_{n=1}^\infty\frac{\cos(2(2n-1)\theta)}{4(2n-1)^2 - 1}\\
	&= \frac{8}{\pi}\sum_{n=1}^\infty\frac{\cos(m_n\theta)}{m_n^2 - 1}
\end{align*}
where $m_n = 4n-2$. From \eqref{eq:Wgeneral_form} it follows that
$W_\theta(r,\theta) = \frac{8}{\pi}\sum_{n=1}^\infty\frac{r^{m_n^2}}{m_n^2 - 1}\cos(m_n\theta)$ and
\[w(r\cos\theta,r\sin\theta) = W(r,\theta) = \frac{8}{\pi}\sum_{n=1}^\infty\frac{r^{m_n^2}}{(m_n^2 - 1)m_n}\sin(m_n\theta).\]
There is no integration constant because the average of $W(1,\theta)$ is zero. This is important. The strategy would not have worked if the hodograph transform is taken over $\Omega_1$ alone. 

When transforming back and letting $\bfg\colon S\to B$ be the inverse of $\nabla w^\top = \bff$ -- that is, $\bfg(\bfx) = \nabla u^\top(\bfx)$. Then $\bff(\bfp) = \bfx$ if and only if $\bfp = \bfg(\bfx)$ and the identity $w(\bfp) = \bfp^\top\bff(\bfp) - u(\bff(\bfp))$ yields
\begin{equation}\label{eq:derive_u}
u(\bfx) = \bfx^\top\bfg(\bfx) - w(\bfg(\bfx)).
\end{equation}
Finally, one may observe that this is a critical value of the function
\[\begin{bmatrix}
r\\ \theta
\end{bmatrix} = \bfr \mapsto f_\bfx(\bfr) := \bfx^\top\Phi(\bfr) - w(\Phi(\bfr)),\qquad \Phi(r,\theta) = \begin{bmatrix}
r\cos\theta\\ r\sin\theta
\end{bmatrix}.\]
Indeed, the critical point of $f_\bfx$ is $\bfr$ such that $\nabla w(\Phi(\bfr)) = \bfx^\top$. Equivalently, $\Phi(\bfr) = \bfg(\bfx)$. Plugging this value back into $f_\bfx$, yields $u(\bfx)$.
When restricted to $0\leq\theta\leq \pi/2$, the graph of $f_\bfx$ is a saddle and I therefore assumed that the critical point is a minimax for every parameter $\bfx\in\Omega_1$.

It is possible that the approach taken can serve as a general method to construct solutions of $\Delta_\infty u = 0$. Note that \eqref{eq:derive_u} should be $\infty$-harmonic for \emph{any} particular solution $w$ of the hodograph equation \eqref{eq:derive_w} provided $\nabla w$ is locally invertible.

\paragraph{Acknowledgements:}
The final draft of the manuscript was written at Institute Mittag-Leffler in Djursholm, Sweden, during the program \emph{Geometric Aspects of Nonlinear Partial Differential Equations} in the fall of 2022. Supported by the Swedish Research Council under the grant no. 2016-06596.

A special thanks goes to Peter Lindqvist.

\bibliographystyle{alpha}
\bibliography{/Users/karlkb/Documents/references.bib}


\end{document}